\documentclass{amsart}
\usepackage{amssymb,amsfonts,amsthm,epsfig}
\usepackage{mathtools}
\usepackage[symbol]{footmisc}
\usepackage{svn}
\usepackage{array}
\usepackage{microtype,booktabs}
\usepackage{CJKutf8}

\usepackage[usenames,dvipsnames]{xcolor}
\usepackage[bookmarksopen,bookmarksdepth=2]{hyperref}
\hypersetup{colorlinks=true,citecolor=NavyBlue,linkcolor=BrickRed,urlcolor=Green}
\usepackage{tikz}
\usetikzlibrary{decorations.pathreplacing}
\usetikzlibrary{calc}
\usetikzlibrary{fit,shapes.geometric,arrows.meta}

\theoremstyle{plain}
\newtheorem{thm}[equation]{Theorem}

\newtheorem{prop}[equation]{Proposition}
\newtheorem{lem}[equation]{Lemma}
\newtheorem{lemma}[equation]{Lemma}
\newtheorem{cor}[equation]{Corollary}

\newtheorem{caution}[equation]{Caution}
\theoremstyle{definition}

\newtheorem{remark}[equation]{Remark}

\newtheorem{defn}[equation]{Definition}
\newtheorem{exercise}[equation]{Exercise}

\numberwithin{equation}{section}
\numberwithin{figure}{subsection}

\newcommand{\C}{\CC}

\newcommand{\Q}{\QQ}
\newcommand{\R}{\RR}
\newcommand{\Z}{\ZZ}

\newcommand{\CC}{{\mathbb C}}

\newcommand{\QQ}{{\mathbb Q}}
\newcommand{\RR}{{\mathbb R}}

\newcommand{\ZZ}{{\mathbb Z}}

\DeclareMathOperator{\GL}{GL}

\DeclareMathOperator{\SL}{SL}

\newcommand{\la}{\langle}
\newcommand{\ra}{\rangle}

\newcommand{\sm}[4]{\left(\begin{smallmatrix} #1 & #2 \\ #3 & #4 \end{smallmatrix}\right)}
\newcommand{\lm}[4]{\begin{pmatrix} #1 & #2 \\ #3 & #4 \end{pmatrix}}

\DeclareMathOperator{\sgn}{sgn}

\newcommand{\genlegendre}[4]{%
  \genfrac{(}{)}{}{#1}{#3}{#4}%
  \if\relax\detokenize{#2}\relax\else_{\!#2}\fi
}
\newcommand{\leg}[3][]{\genlegendre{}{#1}{#2}{#3}}

\newcommand{\tleg}[3][]{\genlegendre{1}{#1}{#2}{#3}}

\begin{document}
\title{An elementary proof of Newman's eta-quotient theorem}
\author{David Savitt}
\address{Department of Mathematics, Johns Hopkins University, Baltimore MD 21218}

\begin{abstract}
Let $\eta(z)$ be the Dedekind eta function. Newman \cite{MR91352,MR107629} studied the modularity of eta-quotients, giving necessary and sufficient conditions for a function of the form $\prod_{0 < m \mid N} \eta(m z)^{r_m}$ to be a (weakly) holomorphic modular form of level $N$. We explain a proof of Newman's theorem. The key observation is that although $\Gamma_1(N)$ is not generated by  $\sm{1}{1}{0}{1}$ and $\sm{1}{0}{N}{1}$, it is generated by those two matrices together with any congruence subgroup. Modularity with respect to some congruence subgroup is established using a simple identity involving the multiplier system of $\eta(z)$, whose proof is elementary in the sense that it avoids the use of Dedekind sums.
\end{abstract}

\maketitle

\section{Introduction}

The Dedekind eta function, defined on the complex upper half-plane by the formula
\[ \eta(z) = q^{\frac{1}{24}} \prod_{n \ge 1} (1 - q^n) \]
with $q = e^{2\pi i z}$ and $q^{\frac{1}{24}}$ understood to mean $e^{2 \pi i z / 24}$, is the source of numerous explicit examples of classical modular forms. To name just a few:\ $\Delta(z) = \eta(z)^{24}$ is the famed cusp form of weight $12$ and level $1$ whose Fourier coefficients are the Ramanujan tau function, $\theta_3(z) = \frac{\eta(2z)^5}{\eta(z)^2 \eta(4z)^2}$ is the theta series $\sum_{n \in \Z} q^{n^2}$, and $\eta(11z)^2 \eta(z)^2$ is the modular form associated to the elliptic curve $y^2-y = x^3-x^2$ of conductor $11$ (which is the minimal conductor of an elliptic curve over $\Q$).

Define an \emph{eta-quotient} to be a function of the form
\[ f(z) = \prod_{0 < m \mid N} \eta(m z)^{r_m} \]
for some positive integer $N$ and integers $r_m$. We set $k = \frac 12 \sum_{m} r_m$, the \emph{weight} of~$f$. Thus each of the three functions mentioned above is an eta-quotient, of weight $12$,~$\frac 12$, and~$2$ respectively. Thanks to the relation \[ q^{\frac{1}{24}}/\eta(z) = \prod_{n\ge 1} (1-q^n)^{-1} = \sum_{n\ge 0} p(n) q^n,\] where $p(n)$ is the partition function, eta-quotients are the foundation of many works in partition theory and representation theory;\ see \cite{MR2020489} for  numerous examples.

In the late 1950s, Morris Newman established the following extremely useful theorem, which governs the modularity properties of eta-quotients.

\begin{thm}[\cite{MR91352,MR107629}]\label{thm:newman}
Let $f$ be as above and suppose $k \in \Z$. Then $f$ is a weakly holomorphic modular form on $\Gamma_1(N)$ if and only if
\[ \sum_{0 < m \mid N} m r_m \equiv 0 \pmod{24} \quad \mathrm{and} \quad \sum_{0 < m \mid N} m r_{N/m} \equiv 0 \pmod{24} .\]
Furthermore the Nebentypus character of $f$ is the character $\chi(d) = \leg{(-1)^k s}{d}$, where $s = \prod_{m} m^{r_{m}}$ and the representative $d$ of a class in $(\Z/N\Z)^{\times}$ is chosen to be odd.
\end{thm}

Here $\leg{m}{d}$ denotes the extended Jacobi symbol, which for positive odd $d$ is extended multiplicatively (in $d$) from the Legendre symbol, and is defined to be $\mathrm{sgn}(m) \leg{m}{|d|}$ when $d < 0$. We review properties of this extended Jacobi symbol in Appendix~\ref{app:jacobi}.

\begin{remark}
 \emph{Weakly holomorphic} means that we do not impose any conditions at the cusps. It is straightforward to check that the $q$-expansion of $\eta(m z)$ at a cusp $\frac{a}{c}$ has leading term $q^{\frac{1}{24} \frac{(c,m)^2}{m}}$, and so for $f$ to be holomorphic at the cusps it is necessary and sufficient that 
\[ \sum_{0 < m \mid N } \frac{(c,m)^2}{m} r_m \ge 0 \]
for all $c \mid N$, with strict inequality in order that $f$ be a cusp form. This observation is often attributed to \cite{MR417060}.
\end{remark}

Theorem~\ref{thm:newman} is perhaps slightly surprising at first glance. As we will see, it is not difficult to check that $\sum_{0 < m \mid N} m r_m \equiv 0 \pmod{24}$ if and only if $f(z)$ has the correct modular transformation property under the matrix $T =\sm{1}{1}{0}{1}$, while  $\sum_{0 < m \mid N} m r_{N/m} \equiv 0 \pmod{24}$  if and only if $f(z)$ has the correct modular transformation property under the matrix $\bar{T}^N = \sm{1}{0}{N}{1}$, with $\bar{T} = \sm{1}{0}{1}{1}$. But in general $T$ and $\bar{T}^N$ do not generate $\Gamma_1(N)$;\ so why should these conditions be sufficient?

Newman proves his theorem via delicate manipulations with the \emph{multiplier system} for~$\eta$, the function $\varepsilon : \SL_2(\Z) \to \mu_{24}(\C)$ defined by the formula $\eta|_{\gamma} = \varepsilon(\gamma) \eta$.
(We will make this definition more precise in the next section.) The general formula for the multiplier system of $\eta$ is somewhat elaborate,  requiring the use of Dedekind sums. If $f$ is an eta-quotient satisfying the conditions of Theorem~\ref{thm:newman}, Newman uses this formula to evaluate $f|_\gamma$ on a general element $\gamma \in \Gamma_1(N)$ whose upper-left entry is relatively prime to $6$, and then applies various Dedekind sum identities to show that $f|_{\gamma} = f$ for such matrices.

While teaching a class on modular forms for talented high school students at \href{https://www.mathcamp.org}{Canada/USA Mathcamp} in July 2025, we attempted to understand Newman's result more conceptually, leading to the proof that we explain below. The key observation is that although $\sm{1}{1}{0}{1}$ and $\sm{1}{0}{N}{1}$ do not generate $\Gamma_1(N)$ in general, they do generate it together with \emph{any} congruence subgroup $\Gamma \subset \SL_2(\Z)$. Thus it suffices to prove that the eta-quotient is modular for \emph{some} congruence subgroup. This is a general fact about eta-quotients of integer weight. To prove this, the only property of $\varepsilon$ that we need is the following. 

\begin{thm}\label{thm:eta-property} If $\gamma = \sm{a}{b}{c}{d} \in \Gamma_0(24) \cap \Gamma^0(24)$ and $c \neq 0$ then
  \[ \eta|_{\gamma} = \leg{c}{d} e^{\pi i (d-1)/4 } \eta.\]
  Here $\Gamma^0(N) = \Gamma_0(N)^t$ denotes the subgroup of matrices that are lower triangular modulo $N$.
\end{thm}

From this formula we deduce that $\eta(z)^2$ and $\eta(mz)/\eta(z)$ are both modular for some congruence subgroup. Since any eta-quotient of integer weight can be written as a product of these, the desired modularity follows.

Finally, it remains to prove Theorem~\ref{thm:eta-property}. We do this by proving a somewhat more general formula, for $\gamma \in \Gamma_0(4)$, whose statement and proof can still be formulated without reference to Dedekind sums. Thus our proof avoids Dedekind sums entirely, and so is in some sense more elementary than the original.

\subsection*{Acknowledgments} We thank Canada/USA Mathcamp for hosting the visit during which this project began, and the Mathcamp students for their enthusiasm. The author is supported by a gift from the Simons Foundation under the Travel Support for Mathematicians program.

\section{The multiplier system for the eta function}\label{sec:multiplier}

The eta function has the classical transformation properties
\[ \eta(z+1) = e^{\pi i/12 }\eta(z), \quad \eta(-1/z) = e^{-\pi i/4} \sqrt{z} \cdot \eta(z), \]
where the branch of $\sqrt{z}$ on the upper half-plane is chosen to take values in the first quadrant. The first property is immediate;\ while the second property can be proved, for example, by observing that the logarithmic derivative $\eta'/\eta$ is a scalar multiple of the Eisenstein series $G_2$ of weight $2$, and applying the corresponding transformation property of~$G_2$.

For each $\gamma = \left(\begin{smallmatrix}
  a & b \\ c & d
\end{smallmatrix}\right) \in \GL^+_2(\R)$ we define $j(\gamma,z) = cz+d$, and we further define $j(\gamma,z)^{\frac12}$ for $z$ in the upper half-plane to be the branch with argument in the range $(-\frac{\pi}{2},\frac{\pi}{2}]$. The standard cocycle formula
\[ j(\gamma'\gamma,z) = j(\gamma,z) j(\gamma',\gamma z) \] implies that
\[ j(\gamma'\gamma,z)^{\frac12} = j(\gamma,z)^{\frac12}  j(\gamma',\gamma z)^{\frac12} \sigma(\gamma',\gamma) \] 
for some $\sigma(\gamma',\gamma) \in \{\pm 1\}$.\footnote{It seems to us that taking the argument of $j(\gamma,z)^{\frac12}$ instead to be in the range $[-\frac{\pi}{2},\frac{\pi}{2})$ could be a somewhat better choice:\ for example our slightly odd definition of the sign of an element of $\SL_2(\Z)$ (Definition~\ref{defn:gamma-sign}) could be rectified. However, changing this convention would also require altering the prevailing convention for the Jacobi symbol $\leg{0}{-1}$, in order to keep the statement of Theorem~\ref{thm:gamma0-4-thm} uniform. For this reason, and to align with existing literature, we hew to the standard choice.}

If $f(z)$ is an eta-quotient of weight $k \in \frac12 \Z$ and $\gamma \in \GL^+_2(\R)$ we write
\[ f|_\gamma  = f(\gamma z) j(\gamma,z)^{-k}, \]
where $\gamma z = \frac{az+b}{cz+d}$ as usual when $\gamma = \sm{a}{b}{c}{d}$. This is the standard ``slash operator'', except that we suppress the weight $k$ from the notation. Then $f|_{\gamma' \gamma} = (f|_{\gamma'})|_{\gamma}$ if $k$ is an integer, and we say that $f$ is modular of level $\Gamma$ if $f|_\gamma = f$ for all $\gamma\in \Gamma$. If instead $k$ is a half-integer then
\begin{equation}
  \label{eq:slash}
  f|_{\gamma'\gamma} = (f|_{\gamma'})|_{\gamma} \cdot \sigma(\gamma',\gamma).
\end{equation}

Returning to the example of the eta function, the classical transformation properties can now be written
\begin{equation}\label{eq:explicit-eta}
  \eta|_T = e^{\pi i/12} \eta, \quad \eta|_S = e^{-\pi i/4} \eta
  \end{equation}
where $T = \sm{1}{1}{0}{1}$ as before, and $S = \sm{0}{-1}{1}{0}$. 

\begin{defn}
  The \emph{multiplier system} of the Dedekind eta function is the function $\varepsilon : \SL_2(\Z) \to \mu_{24}(\C)$ defined by
  \[ \eta|_{\gamma} = \varepsilon(\gamma) \eta. \]
  Here $\mu_{24}(\C)$ is the group of $24$th roots of unity in $\C$.
\end{defn}

To see that $\varepsilon$ is valued in $\mu_{24}$, note that since $\eta$ has weight $\frac{1}{2}$, formula \eqref{eq:slash} implies that \[\varepsilon(\gamma' \gamma) = \varepsilon(\gamma')\varepsilon(\gamma) \sigma(\gamma',\gamma) \in \{ \pm  \varepsilon(\gamma')\varepsilon(\gamma) \}.\]
The formulas \eqref{eq:explicit-eta} give $\varepsilon(S), \varepsilon(T) \in \mu_{24}$. Since  $\SL_2(\Z)$ is generated by $S$ and $T$ the claim follows.

\section{Proof of Newman's theorem}

Granting Theorem~\ref{thm:eta-property}, we now prove Newman's theorem.  Write $ w_N = \sm{0}{-1}{N}{0}.$

\begin{prop}\label{prop:basic-eta-quotient-slash}
Let $f(z) = \prod_{0 < m \mid N } \eta(mz)^{r_m}$ be an eta-quotient of weight $k$. Define $f^* = \prod_{0 < m \mid N} \eta(mz)^{r_{N/m}}$. Then we have the formulas

\[ f|_T = e^{\pi i b/12} f\]
where $b = \sum_{0 < m \mid N} m r_m$,
and
  \[ f|_{w_N} = \frac{e^{-\pi i k/2}}{\sqrt{s}} f^* \]
where $s = \prod_{0 < m \mid N} m^{r_m}$.
\end{prop}

\begin{proof}
  Write $\eta_m(z) = \eta(mz)$. For the first formula we compute
  \[ \eta_m|_T = \eta_m(z+1) = \eta(mz + m) = e^{\pi i m/12} \eta(mz) = e^{\pi i m/12} \eta_m(z),\]
  and therefore $f|_T = \exp(\pi i (\sum_m m r_m) /12 ) f$.

  For the second part, we compute that
  \begin{align*}
   \eta_m|_{w_N}(z) & = (Nz)^{-\frac12}\eta_m\left(-\frac{1}{Nz}\right) \\
     & = (Nz)^{-\frac12}\eta\left(-\frac{m}{Nz}\right) \\
     & = (Nz)^{-\frac12} e^{-\pi i/4} \sqrt{\frac{Nz}{m}} \cdot  \eta\left(\frac{N}{m} z\right)\\
     & = \frac{e^{-\pi i/4}}{\sqrt{m}} \eta_{N/m}(z).
  \end{align*}
  This gives the claim for $\eta_m$, and the claim for $f$ follows by multiplicativity.
  \end{proof}

\begin{cor}\label{cor:eta-quotient-necessary}
  Let $f(z) = \prod_{0 < m \mid N } \eta(mz)^{r_m}$ be an eta-quotient. Then we have
\begin{equation}\label{eq:T-cond}
 f|_{\sm{1}{1}{0}{1}} = f 
 \quad  \iff \quad 
 \sum_{0 < m \mid N} m r_m \equiv 0 \pmod{24}    
\end{equation}
and
\begin{equation}\label{eq:wN-cond}
  f|_{\sm{1}{0}{N}{1}}= f \quad \iff \quad \sum_{0 < m \mid N} m r_{N/m} \equiv 0 \pmod{24}. 
\end{equation}
If $2,3 \nmid N$ then the conditions \eqref{eq:T-cond} and \eqref{eq:wN-cond} are equivalent because $m^2 \equiv 1 \pmod{24}$ for all $m$ prime to $24$.
\end{cor}

\begin{proof}
  The first equivalence is immediate from the first part of Proposition~\ref{prop:basic-eta-quotient-slash}. Observe that \[\lm{1}{0}{-N}{1} = w_N \lm{1}{1}{0}{1} w^{-1}_N \]
and so by the second part of Proposition~\ref{prop:basic-eta-quotient-slash} we have $f|_{\sm{1}{0}{N}{1}} = f$ if and only if $f^*|_{\sm{1}{1}{0}{1}} = f^*$. Thus the second equivalence follows from the first equivalence applied to $f^*$.
\end{proof}

Corollary~\ref{cor:eta-quotient-necessary} establishes that Newman's conditions are necessary if $f$ is to transform as a modular form under $\Gamma_1(N)$.  This is the easy direction of the theorem.  We now address the more difficult direction. Recall, again, what we want to prove.

\begin{thm}\label{thm:what-we-want} Suppose $k \in \Z$. If $f|_{\sm{1}{1}{0}{1}}= f$ and $f|_{\sm{1}{0}{N}{1}} = f$ then 
  $f$ is a weakly holomorphic modular form on $\Gamma_1(N)$ with Nebentypus character $\chi(d) = \leg{(-1)^k s}{d}$, where $s = \prod_{m} m^{r_{m}}$.
\end{thm}

We have the following key calculation.

\begin{lem}\label{lem:basic-eta-quotients} The function $\eta(z)^2$ is modular of level $\Gamma_0(24) \cap \Gamma^0(24) \cap \Gamma(4)$ while for $m \ge 1$ the quotient $\eta(mz)/\eta(z)$ is modular of level $\Gamma_0(24m) \cap \Gamma^0(24) \cap \Gamma(4m)$. More precisely, we have: 
\begin{enumerate}
  \item If $\gamma = \sm{a}{b}{c}{d} \in \Gamma_0(24) \cap \Gamma^0(24)$, we have
  \[ \eta^2 |_\gamma = \leg{-1}{d} \eta^2 . \]

  \item If $\gamma = \sm{a}{b}{c}{d} \in \Gamma_0(24m) \cap \Gamma^0(24)$, we have
  \[ (\eta(mz)/\eta(z))|_\gamma = \leg{m}{d} (\eta(mz)/\eta(z)). \]  
\end{enumerate}
  
\end{lem}

\begin{proof} The claims (1) and (2) imply the first part since $\leg{-1}{d} = 1$ if $d \equiv 1 \pmod{4}$ and $\leg{m}{d} = 1$ if $d \equiv 1 \pmod{4m}$ (see (3),(7) of Prop.~\ref{prop:allow-negative}).

  If $c \neq 0$ then (1) follows from Theorem~\ref{thm:eta-property} and the identity $\leg{-1}{d} = (-1)^{(d-1)/2}$, which holds for all odd $d \in \Z$ by Proposition~\ref{prop:allow-negative}(3). If $c = 0$ then $d = \pm 1$ and the identity can be checked directly.

 For claim (2) the case $c=0$ is similarly straightforward. Suppose $c \neq 0$. Since $\eta(mz) = \eta(z)|_{\sm{m}{0}{0}{1}}$ and \[ \sm{m}{0}{0}{1}\gamma \sm{m}{0}{0}{1}^{-1} = \sm{a}{bm}{c/m}{d} \in \Gamma_0(24) \cap \Gamma^0(24),\]
  we use Theorem~\ref{thm:eta-property} to compute
  \[ \eta(mz)|_{\gamma} = (\eta(z)|_{\sm{m}{0}{0}{1}\gamma \sm{m}{0}{0}{1}^{-1}})|_{\sm{m}{0}{0}{1}} = \leg{c/m}{d} e^{\pi i (d-1)/4} \eta(mz), \]
  and (2) follows on comparing with Theorem~\ref{thm:eta-property} for $\eta(z)|_{\gamma}$. 
\end{proof}

\begin{cor}\label{cor:intermediate-cor}
  Any eta-quotient  $f(z) = \prod_{0 < m \mid N} \eta(mz)^{r_m}$ of integer weight $k$ is (weakly) modular of level $\Gamma_0(24N) \cap \Gamma^0(24) \cap \Gamma(4N)$. Furthermore if $\gamma \in \Gamma_0(24N) \cap \Gamma^0(24)$ then
\[ f|_{\gamma} = \leg{(-1)^k s}{d} f \]
where $s = \prod_{m} m^{r_m}$.
\end{cor}

\begin{proof}
  Indeed, we may rewrite 
\[ f(z) = \eta(z)^{2k} \prod_{1 < m \mid N} (\eta(mz)/\eta(z))^{r_m} \]
and applying Lemma~\ref{lem:basic-eta-quotients} we conclude that $f(z)$ is modular of level $\cap_{m \mid N} (\Gamma_0(24m) \cap \Gamma^0(24) \cap \Gamma(4m)) = \Gamma_0(24N) \cap \Gamma^0(24) \cap \Gamma(4N)$. The claim about the Nebentypus follows by multiplicativity from Lemma~\ref{lem:basic-eta-quotients}(1),(2).
\end{proof}

To deduce that an eta-quotient $f$ as in Theorem~\ref{thm:what-we-want} is modular of level $\Gamma_1(N)$, it remains to check the following.

\begin{prop}\label{prop:congruence-subgroup}
  Any congruence subgroup $\Gamma$ of $\SL_2(\Z)$ that contains $\sm{1}{1}{0}{1}$ and $\sm{1}{0}{N}{1}$ contains $\Gamma_1(N)$.
\end{prop}

\begin{proof}
Suppose $\Gamma \supset \Gamma(M)$ for some integer $M$. Enlarging $M$ if necessary we may assume $N \mid M$. Writing $M = \delta N$, since $\Gamma(\delta N)$ and $T$ generate $\Gamma_1(\delta N)$, we see that $\Gamma \supset \Gamma_1(\delta N)$. It now suffices to prove (for all primes $p$ and positive integers $N$) that $\sm{1}{0}{N}{1}$ and $\Gamma_1(pN)$ together generate $\Gamma_1(N)$. We can then apply this iteratively to remove one prime at a time from $\delta$ until we obtain $\Gamma \supset \Gamma_1(N)$.

 Our claim is equivalent to showing that the image of $\Gamma_1(N)$ in $\SL_2(\Z/pN\Z)$ is contained in the image of $\la \sm{1}{0}{N}{1}, \Gamma_1(pN) \ra$. If $\gamma \in \Gamma_1(N)$, its image $\overline{\gamma}$ in $\SL_2(\Z/pN\Z)$ has the shape
 \[ \lm{  1 + aN }{b}{ cN }{1+dN} \]
 with $a,c,d \in \Z/p\Z$. First suppose that $c = 0$. Then $u := 1 + dN$ is invertible in $\Z/pN\Z$ and $\overline{\gamma} = \lm{ u^{-1} }{0}{ 0 }u \lm{  1 }{ub}{0}{1}$, and so it suffices to compute that
  \[ \lm{u^{-1}}{0}{0}{u}  = \lm{1}{0}{-uN}{1}\lm{1}{-d  + u^{-1} d^2 N}{0}{1} \lm{1}{0}{N}{1} \lm{1}{d}{0}{1}.   \]
 If instead $c \neq 0$, let $c^{-1} \in \Z$ be any inverse of $c$ modulo $p$. Then 
  \[  \lm{  1 + aN }{b}{ cN }{1+dN} = \lm{1}{ac^{-1}}{0}{1} \lm{1}{0}{cN}{1}  \lm{1}{b'}{0}{1}\] for a suitable choice of $b'$, and so again we are done.
\end{proof}

Summarizing, under the hypotheses of Theorem~\ref{thm:what-we-want} we have proved that $f$ is modular of level $\Gamma_1(N)$, and that 
if $\gamma \in \Gamma_0(24N) \cap \Gamma^0(24)$ then
\[ f|_{\gamma} = \tleg{(-1)^k s}{d} f. \]
The Nebentypus of $f$ is the character $\chi : (\Z/N\Z)^{\times} \to \C^{\times}$ defined by the formula $f|_{\sm{a}{b}{c}{d}} = \chi(\overline{d}) f$ for any $\sm{a}{b}{c}{d} \in \Gamma_0(N)$ with $d$ lifting $\overline{d} \in \Z/N\Z$. Suppose $d \in \Z$ is any lift of $\overline{d} \in (\Z/N\Z)^{\times}$ that is relatively prime to $6$. Let $a$ be any inverse of $d$ modulo $24^2 N$, and select $b,c$ with $bc = ad-1$ and $24 \mid b$, $24N \mid c$. Then $\sm{a}{b}{c}{d} \in \Gamma_0(24N) \cap \Gamma^0(24) \subset\Gamma_0(N)$ with $d$ lifting $\overline{d}$, and $\chi(\overline{d})= \leg{(-1)^k s}{d}$ by Corollary~\ref{cor:intermediate-cor}. Thus the Nebentypus of $f$ is as claimed in Theorem~\ref{thm:newman}, except that so far we are restricted to computing $\chi(\overline{d})$ using lifts $d$ with $(d,6) = 1$ rather than just $(d,2) = 1$. This is not an issue if $3 \mid N$, so assume $3 \nmid N$.

Note that $\leg{(-1)^k s}{d} = \leg{(-1)^k s'}{d}$ where $s'$ is the squarefree part of $s$. It follows by Proposition~\ref{prop:allow-negative}(7) that
 $\leg{(-1)^k s}{d}$ is periodic in $d$ with period dividing $4N$. If $3 \mid d$, then $3 \nmid d+4N$, and $\leg{(-1)^k s}{d} = \leg{(-1)^k s}{d+4N}$. So computing $\leg{(-1)^k s}{d}$ for odd lifts of $\overline{d}$ with $3 \mid d$ gives the same answer as for odd lifts not divisible by $3$, and we are done.

\section{Proof of the multiplier system identity}

For our exposition to be complete, we must prove Theorem~\ref{thm:eta-property}. In fact Theorem~\ref{thm:eta-property} is an immediate corollary of the following more general result whose formulation  we learned from \cite{MR2070444}. 

\begin{thm}\label{thm:gamma0-4-thm}
 Define $E(a,b,c,d) =  ac(1-d^2) + d(b-c+3) - 3.$ Then for each $\gamma = \sm{a}{b}{c}{d} \in \Gamma_0(4)$ we have
  \[  \varepsilon(\gamma) =  \leg{c}{d} e^{\pi i E(a,b,c,d) / 12 }. \]
\end{thm}

Recall that we take $\leg{c}{d} = \sgn(c) \leg{c}{|d|}$ if $d < 0$, and conventionally $\leg{0}{\pm 1} = 1$.

\begin{remark}
  In fact the theorem holds verbatim for $\gamma \in \Gamma_0(2)$, but we will give the details only for $\Gamma_0(4)$, as the proof for $\Gamma_0(4)$ involves somewhat less case work;\ however, see Exercises~\ref{ex:gamma-2-sub-1} and~\ref{ex:gamma-2-final} for a sketch of the missing details for $\Gamma_0(2)$. 

  Since $\Gamma_0(2)$ has finite index in $\SL_2(\Z)$, there is a sense in which Theorem~\ref{thm:gamma0-4-thm} actually computes the whole multiplier system:\ namely, choose left coset representatives $\gamma_1,\gamma_2,\gamma_3$ for $\Gamma_0(2)$, and compute $\varepsilon(\gamma_i)$ for each. Then, writing a general element $\gamma \in \SL_2(\Z)$ as $\gamma_i \gamma'$ with $\gamma' \in \Gamma_0(2)$, we have $\varepsilon(\gamma) = \varepsilon(\gamma_i) \varepsilon(\gamma') \sigma(\gamma_i,\gamma')$.
\end{remark}

We turn to the proof of Theorem~\ref{thm:gamma0-4-thm}. Although $\Gamma_1(N)$ is not generated by  $\sm{1}{1}{0}{1}$ and $\sm{1}{0}{N}{1}$ in general, it is true for $N \le 4$ by a standard Euclidean algorithm argument.
Thus $\Gamma_0(4)$ is generated by $T = \sm{1}{1}{0}{1}$, $\bar{T}^4 = \sm{1}{0}{4}{1}$, and $-I = \sm{-1}{0}{0}{-1}$. So, to prove the theorem, it suffices to check the theorem for $T$, $\bar{T}^4$, and $-I$, and to show that $\varepsilon$ as defined in Theorem~\ref{thm:gamma0-4-thm} satisfies the multiplicative relation
\begin{equation}\label{eq:multip} \varepsilon(\gamma'\gamma)  =  \varepsilon(\gamma') \varepsilon(\gamma) \sigma(\gamma', \gamma). 
\end{equation}
Moreover to check the multiplicative relation \eqref{eq:multip} it suffices to consider  the case where $\gamma'$ is one of our three generators $T, \bar{T}^4, -I$. We now carry out this plan.

\begin{lem}
  The theorem holds for $T$, $\bar{T}^4$, and $-I$;\ namely we have $\varepsilon(T) = e^{\pi i/12}$, $\varepsilon(-I) = -i$, and $\varepsilon(\bar{T}^4) = e^{-\pi i/3}$.
\end{lem}

\begin{proof} 
  We have already seen $\varepsilon(T) = e^{\pi i/12}$. Since 
  \[ \eta(z) = \eta( (-I) \cdot z ) = \varepsilon(-I) (-1)^{\frac12} \eta(z) \]
  where $(-1)^{\frac 12}$ has been chosen to be $i$, we find $\varepsilon(-I) = -i$ as claimed. For $\bar{T}^4$ we apply the transformation formula for $\eta(-1/z)$ with $-(4z+1)/z$ substituted for $z$ to obtain
\begin{align*}
  \eta\left(\frac{z}{4z+1}\right) & = e^{-\pi i/4} \left(-\frac{4z+1}{z}\right)^{\frac 12} \eta\left(\frac{-4z-1}{z}\right) \\
  & = e^{-\pi i /4} e^{-\pi i /3} \left(-\frac{4z+1}{z}\right)^{\frac 12} \eta\left(-\frac{1}{z}\right)\\
  & = e^{-\pi i /2} e^{-\pi i / 3 }\left(-\frac{4z+1}{z}\right)^{\frac 12}  z^{\frac12} \eta(z).
\end{align*}
We have chosen branches of square roots so that all three of $z^{\frac 12}$, $\left(-\frac{4z+1}{z}\right)^{\frac 12}$, and $(4z+1)^{\frac 12}$ lie in the first quadrant. Therefore $z^{\frac 12} \cdot \left(-\frac{4z+1}{z}\right)^{\frac 12}$ lies in the upper half-plane, and so we must have $z^{\frac 12} \cdot \left(-\frac{4z+1}{z}\right)^{\frac 12} e^{-\pi i/2} = (4z+1)^{\frac 12}$ (since both sides have positive real part while $-(4z+1)^{\frac 12}$ has negative real part). Therefore $\eta(z/(4z+1)) = e^{-\pi i/3} (4z+1)^{\frac 12} \eta(z)$ and $\varepsilon(\bar{T}^4) = e^{-\pi i/3}$ as claimed.
\end{proof}

Before we address the multiplicative relation \eqref{eq:multip}, we must make the sign $\sigma(\gamma',\gamma)$ explicit.

\begin{defn}\label{defn:gamma-sign} Suppose $\gamma  = \sm{a}{b}{c}{d} \in \SL_2(\Z)$. We write
 \[\begin{cases}
   \gamma > 0 & \text{if } c>0 \text{ or } c=0,d < 0\\
   \gamma < 0 & \text{if } c<0 \text{ or } c=0,d > 0.\\
 \end{cases}\]
 We will also write $\sgn(\gamma)=1$ if $\gamma>0$ and $\sgn(\gamma) = -1$ if $\gamma < 0$. 
\end{defn}

In particular $I < 0$ and $-I > 0$. Making the opposite choice would complicate the statement of the following lemma. This is an artifact of having chosen $j(-I,z)^{\frac 12} = i$ rather than $j(-I,z)^{\frac 12} = -i$.

\begin{lemma}\label{lem:sign-cocycle}
   We have  $\sigma(\gamma',\gamma) = -1$ if either (i)  $\gamma',\gamma > 0$ and $\gamma'\gamma < 0$, or else (ii) $\gamma',\gamma < 0$ and $\gamma'\gamma > 0$. We have $\sigma(\gamma',\gamma) = 1$ otherwise.
\end{lemma}

\begin{proof}
  We have $\gamma > 0$ if and only if $cz + d$ takes values in the upper half-plane or negative real axis, and $\gamma < 0$ if and only if $cz+d$ takes values in the lower half-plane or positive real axis. Thus $j(\gamma,z)^{\frac12}$ has argument in $(0,\frac{\pi}{2}]$ if $\gamma > 0$, and argument in $(-\frac{\pi}{2}, 0]$ if $\gamma < 0$. 

  Then, for example, if $\gamma,\gamma'> 0$ the argument of $j(\gamma,z)^{\frac 12} j(\gamma',\gamma z)^{\frac 12}$ lies in $(0,\pi]$;\ so $j(\gamma,z)^{\frac 12} j(\gamma',\gamma z)^{\frac 12}$ will fail to equal $j(\gamma'\gamma ,z)^{\frac 12}$ precisely when the latter has argument in $(-\frac{\pi}{2},0]$, i.e., when $\gamma\gamma' < 0$. The other cases are similar.
\end{proof}

\begin{caution}
  For the rest of this section $\varepsilon$ denotes the function defined in the statement of Theorem~\ref{thm:gamma0-4-thm}, which will be proved to be the multiplier system of $\eta$ once we check that it satisfies \eqref{eq:multip}.
\end{caution}

\begin{lemma}
  The multiplicative relation \eqref{eq:multip} holds if $\gamma' = -I$ or $T$.
\end{lemma}

\begin{proof}
First consider $\gamma' = -I$.  Since $\gamma' > 0$, and since $\gamma,-\gamma$ have opposite signs, we see $\sigma(-I,\gamma) = -\sgn(\gamma)$. Thus the target \eqref{eq:multip} becomes 
\[ \varepsilon(-\gamma) = i \sgn(\gamma) \varepsilon(\gamma) . \]
We have $E(-a,-b,-c,-d) = E(a,b,c,d) - 6d$, and therefore
\[  \varepsilon(-\gamma) =   \tleg{-c}{-d} e^{\pi i E(-a,-b,-c,-d) / 12 } = \tleg{-c}{-d} \tleg{c}{d} e^{- \pi  i d/2 } \varepsilon(\gamma).\]

If $c\neq 0$ we have $\sgn(\gamma) = \sgn(c)$ and 
\[ \tleg{-c}{-d}\tleg{c}{d} = (-1)^{\frac{-d-1}{2}} \tleg{c}{d} \tleg{c}{-d} = (-1)^{\frac{-d-1}{2}} \sgn(c)  = (-1)^{\frac{d+1}{2}} \sgn(\gamma).\]
Then 
\[  \varepsilon(-\gamma) = e^{\pi i (d+1)/2 } e^{-\pi i d/2} \sgn(\gamma) \varepsilon(\gamma) = i \sgn(\gamma) \varepsilon(\gamma) \] as desired. If instead $c = 0$ and $d = \pm 1$ then $\leg{-c}{-d}\leg{c}{d} = 1$, $e^{-\pi i d/2 } = - i \sgn(d)$ and $\sgn(\gamma) = -\sgn(d)$, and again the result follows.

Now suppose $\gamma' = T$. If $\gamma = \sm{a}{b}{c}{d}$ then $T \gamma = \sm{a+c}{b+d}{c}{d}$. Since $\sgn(\gamma) = \sgn(T\gamma)$, we have $\sigma(T,\gamma) = 1$ in all cases, and the target relation is
\[ \varepsilon(T\gamma) = e^{\pi i/12} \varepsilon(\gamma). \]
We have
\[ E(a+c,b+d,c,d) - E(a,b,c,d) =  c^2(1-d^2) + d^2 \]
and therefore
\begin{align*}
  \varepsilon(T\gamma) & = \leg{c}{d} e^{\pi i E(a+c,b+d,c,d) / 12 }\\
  & = \varepsilon(\gamma) \cdot e^{\pi i (c^2(1-d^2) + d^2)/12} .   
\end{align*}
But \[ c^2(1-d^2) + d^2 = (c^2-1)(1-d^2) + 1 \equiv 1 \pmod{24} \]
because $d$ is odd, so $1-d^2\equiv 0 \pmod 8$;\ and at least one of $c,d$ is not divisible by $3$, so $(c^2-1)(1-d^2)$ is divisible by $3$. Therefore $\varepsilon(T\gamma) = \varepsilon(\gamma) e^{\pi i /12}$ as desired.
\end{proof}

It remains to check the relation \eqref{eq:multip} for $\gamma' = \bar{T}^4$. We have
\[ \lm{1}{0}{4}{1} \lm{a}{b}{c}{d} = \lm{a}{b}{c+4a}{d+4b} =: \lm{a}{b}{c'}{d'}.\]
We begin with the following two observations.

\begin{lem}\label{lem:jac-sym-trick}
  If $a > 0$ then we have $\leg{c'}{d'} = \leg{c}{d}$.
  \end{lem}

\begin{proof} First suppose $c = 0$. Then $\gamma = \sm{1}{b}{0}{1}$ and $\bar{T}^4 \gamma = \sm{a}{b}{c'}{d'} = \sm{1}{b}{4}{4b+1}$, and we see directly that $\leg{c}{d} = \leg{c'}{d'} = 1$. If instead $c' = 0$ then $\gamma = \sm{1}{b}{-4}{-4b+1}$  and $\bar{T}^4 \gamma = \sm{1}{b}{0}{1}$, and again  $\leg{c}{d} = \leg{c'}{d'} = 1$. Finally if $b = 0$ then $a = d = d' = 1$ and the claim is trivial.

Now suppose $c,c' \neq 0$ and $b \neq 0$.  Since $(b,d) = 1$ we may write
  \[  \tleg{c}{d} = \tleg{bc}{d}\tleg{b}{d} = \tleg{ad-1}{d} \tleg{b}{d}.\]
 If $d < 0$ then $ad-1 < 0$ since we have assumed $a > 0$. Therefore irrespective of the sign of $d$ we have  $\leg{ad-1}{d} = \leg{-1}{d}$ by Property $(6)'$ of Proposition~\ref{prop:allow-negative}, and
  \[  \tleg{c}{d} =  \tleg{-1}{d} \tleg{b}{d} =  \tleg{-b}{d}.\]
The identical argument shows that $\leg{c'}{d'} = \leg{-b}{d'}$. But $d' = d + 4b$,
and $\leg{-b}{d+4b} = \leg{-b}{d}$ by Property (7) of Proposition~\ref{prop:allow-negative}. The claim follows.
\end{proof}

\begin{exercise}\label{ex:gamma-2-sub-1}
  Lemma~\ref{lem:jac-sym-trick} uses that we are working in $\Gamma_0(4)$ rather than $\Gamma_0(2)$. Suppose instead that $\sm{a}{b}{c'}{d'} = \sm{1}{0}{2}{1}\sm{a}{b}{c}{d}$ with $a > 0$. Prove that $\leg{c'}{d'} = \leg{c}{d}$ if $b \equiv 0,3 \pmod{4}$ and $\leg{c'}{d'} = -\leg{c}{d}$ if $b \equiv 1,2 \pmod{4}$.
\end{exercise}

\begin{lem}\label{lem:sign-swap}
  If the relation \eqref{eq:multip} holds for a pair $\gamma',\gamma$ then it also holds for the pair $\gamma', -\gamma$.
\end{lem}

\begin{proof}
 Assume \eqref{eq:multip} holds for the pair $\gamma',\gamma$;\ that is, assume 
 \[ \varepsilon(\gamma' \gamma) = \varepsilon(\gamma') \varepsilon(\gamma) \sigma(\gamma', \gamma). \]
 We want to prove that
 \[ \varepsilon(\gamma' (-\gamma)) = \varepsilon(\gamma') \varepsilon(-\gamma) \sigma(\gamma', -\gamma). \]
 Substituting $\varepsilon(-\gamma) = i \sgn(\gamma) \varepsilon(\gamma)$ and $\varepsilon(-\gamma' \gamma) = i \sgn(\gamma' \gamma)\varepsilon(\gamma'\gamma)$ and comparing the equations, we are reduced to showing that
 \[ \sigma(\gamma', -\gamma) \sigma(\gamma',\gamma) = \sgn(\gamma) \sgn(\gamma' \gamma) .\]
Both sides can be checked to be $1$ if $\gamma, \gamma'\gamma$ have the same sign and $-1$ if they have the opposite sign.
\end{proof}

Thanks to Lemma~\ref{lem:sign-swap} we are reduced to showing the following.

\begin{prop}
  \label{prop:the-last-hurrah}
  The relation \eqref{eq:multip} holds for $\gamma' = \bar{T}^4$ and $\gamma = \sm{a}{b}{c}{d}$ with $a > 0$.
\end{prop}

\begin{proof}  Since $a > 0$, if $c \ge 0$ then $c' = c + 4a > 0$. Therefore we cannot have $\gamma > 0$ and $\bar{T}^4 \gamma < 0$. It follows that $\sigma(\bar{T}^4,\gamma) = 1$ and we have to prove that
  \[ \varepsilon(\bar{T}^4 \gamma) = \varepsilon(\bar{T}^4)\varepsilon(\gamma)  = e^{-\pi i /3} \varepsilon(\gamma).\]




By Lemma~\ref{lem:jac-sym-trick} we have $\leg{c'}{d'} = \leg{c}{d}$ and so our problem is to show that
\[ E(a,b,c',d') - E(a,b,c,d) \equiv - 4\pmod{24} .\]
We analyze this congruence separately mod $8$ and mod $3$.  Since $1 - d^2 \equiv 1 - (d')^2 \equiv 0 \pmod{8}$ the mod $8$ congruence becomes
\[ (d + 4b)(b-c -4a +3) - d(b-c+3)\equiv   - 4 \pmod{8}.\]
Since $4 - 4ad = -4bc$ this is equivalent to 
\[ 4b(b-2c-4a+3) \equiv 0 \pmod{8}. \]
But either $b$ or $b-2c-4a+3$ is even, and so this congruence holds.  For the mod $3$ congruence, evaluating $E(a,b,c',d') - E(a,b,c,d)$ and replacing all instances of $ad$ with $bc+1$ gives 
\[  E(a,b,c',d') - E(a,b,c,d) \equiv -(a^2-1)(b^2-1) -1 \pmod{3}. \]
Since at least one of $a,b$ is not divisible by $3$ the congruence holds.
\end{proof}

\begin{exercise}\label{ex:gamma-2-final}
To prove Theorem~\ref{thm:gamma0-4-thm} with $\Gamma_0(2)$ in place of $\Gamma_0(4)$ it suffices to show that the relation \eqref{eq:multip} holds for $\gamma' = \bar{T}^2$ and $\gamma = \sm{a}{b}{c}{d}$ with $a > 0$. Just as in the previous proof we have $\sigma(\overline{T}^2,\gamma) = 1$. Set $c' = c+2a$, $d' = d + 2b$.
\begin{enumerate}
  \item Check that $\varepsilon(\bar{T}^2) = e^{-\pi i /6}$.
  \item Use Exercise~\ref{ex:gamma-2-sub-1} to reduce the relation $\varepsilon(\bar{T}^2 \gamma) = e^{-\pi i /6} \varepsilon(\gamma)$ to the congruence
  \[ E(a,b,c',d') - E(a,b,c,d) \equiv \begin{cases} -2 \pmod{24} & \text{if } b \equiv 0,3 \pmod{4}\\
 10 \pmod{24} & \text{if } b \equiv 1,2 \pmod{4} . \end{cases}\]
 \item Confirm the congruence of the previous part.
\end{enumerate}
\end{exercise}

\appendix

\section{The extended Jacobi symbol}\label{app:jacobi}

We collect here the Jacobi symbol identities used in the rest of the text, extended to negative ``denominators''.

\begin{defn} Let $d$ be an odd integer.
    If $d$ is positive, let $\leg{m}{d}$ denote the usual Jacobi  symbol. In particular, by convention $\leg{0}{1} = 1$. If $d$ is negative, we define $\leg{m}{d} = \sgn(m) \leg{m}{|d|}$ except that by convention we take $\leg{0}{-1} = 1$. This extension is the same as the \emph{Kronecker symbol} restricted to odd values of~$d$. 
\end{defn}

We recall that the usual Jacobi symbol has the following basic properties.

\begin{lemma}\label{lem:jac} Suppose $m,m' \in \Z$ and $d,d'$ are positive and odd. Then:
  \begin{enumerate}
    \item $\leg{mm'}{d} = \leg{m}{d}\leg{m'}{d}$.

    \item $\leg{m}{dd'} = \leg{m}{d}\leg{m}{d'}$.

    \item $\leg{-1}{d} = (-1)^{\frac{d-1}{2}}$.

    \item $\leg{2}{d} = (-1)^{\frac{d^2-1}{8}}$.

    \item $\leg{d}{d'}\leg{d'}{d} = (-1)^{\frac{d-1}{2} \frac{d'-1}{2}}$ provided $d,d'$ are coprime.

    \item If $m \equiv m' \pmod{d}$ then $\leg{m}{d} = \leg{m'}{d}$.

    \item If $d \equiv d' \pmod{4m}$ then $\leg{m}{d} = \leg{m}{d'}$.
  \end{enumerate}
\end{lemma}

If we allow negative values of $d$ and $d'$, then we have the following.

\begin{prop}\label{prop:allow-negative}
  Suppose $m,m' \in \Z$ and $d,d'$ are odd. Properties (2)-(4) and (7) of Lemma~\ref{lem:jac} extend verbatim. The other properties change as follows.
\begin{enumerate}
    \item[$(1)'$] $\leg{mm'}{d} = \leg{m}{d}\leg{m'}{d}$ unless $d=-1$ and $m = 0$, $m' < 0$ or $m < 0$, $m'=0$.
    \item[$(5)'$] $\leg{d}{d'}\leg{d'}{d} = (-1)^{\frac{d-1}{2} \frac{d'-1}{2}}$ if and only if at least one of $d,d'$ is positive, provided $d,d'$ are coprime.

  \item[$(6)'$] Suppose $m \equiv m' \pmod{d}$. If $d > 0$ then $\leg{m}{d} = \leg{m'}{d}$. If $d<0$ and $m,m'$ have the same sign, then $\leg{m}{d} = \leg{m'}{d}$.
\end{enumerate}
\end{prop}

\begin{proof}
  Checking that (1)--(4) go over as explained is straightforward. For example, if $d < 0$ then $\leg{-1}{d} = -\leg{-1}{-d} = -(-1)^{\frac{-d-1}{2}} = (-1)^{\frac{d-1}{2}}$, proving (3). We leave (1), (2), (4) to the reader. Property $(6)'$ is similarly straightforward since if $d < 0$ then  
  \[  \sgn(m) \tleg{m}{d} = \tleg{m}{|d|} = \tleg{m'}{|d|} = \sgn(m') \tleg{m'}{d} .\]

  To check $(5)'$ we can assume $d < 0$ and consider separately the cases $d' > 0$ and $d' < 0$. If $d' > 0$ then
  \[ \tleg{d}{d'}\tleg{d'}{d} = (-1)^{\frac{d'-1}{2}} \tleg{|d|}{d'} \tleg{d'}{|d|} = (-1)^{\frac{d'-1}{2}}(-1)^{\frac{d'-1}{2}\frac{-d-1}{2} } = (-1)^{\frac{d'-1}{2}\frac{d-1}{2} }.\] If instead $d' < 0$ then
   \[ \tleg{d}{d'}\tleg{d'}{d} = (-1)^{\frac{d'-1}{2} + \frac{d-1}{2}} \tleg{|d|}{|d'|} \tleg{|d'|}{|d|} =  - (-1)^{\frac{d'-1}{2}\frac{d-1}{2} }.\] 

Finally we check (7). If $d \equiv d' \pmod{4m}$ then $d \equiv d' \pmod{4m'}$ for any divisor $m' \mid m$;\ so by $(1)'$ it suffices to prove the result separately for $m = -1$, $m=2$, and $m$ positive and odd. The first two cases are immediate from (3), (4). Suppose that $m$ is positive and odd. Then
\[ \tleg{m}{d} = \tleg{d}{m} (-1)^{\frac{d-1}{2}\frac{m-1}{2}} = \tleg{d'}{m}  (-1)^{\frac{d'-1}{2}\frac{m-1}{2}}  = \tleg{m}{d'} \]
where the first and third equalities are property $(5)'$ (valid since $m > 0$), and the middle equality uses property (6) to deduce $\leg{d}{m} = \leg{d'}{m}$ (again noting $m > 0$), along with $d' \equiv d \pmod{4}$ to see that the powers of $-1$ are the same.
    \end{proof}

\section*{Data availability}

Data sharing is not applicable to this article.

\bibliographystyle{math} 
\bibliography{nt}

\end{document}